\pdfoutput=1 
\documentclass[11pt, a4paper, english]{amsart}
\usepackage{amsmath} 
\usepackage{amsthm} 
\usepackage{thmtools, thm-restate}
\usepackage{amssymb} 
\usepackage[dvipsnames]{xcolor}

\usepackage{amscd} 
\usepackage{mathtools}
\usepackage{subcaption}
\usepackage{array} 
\usepackage[backref=page,linktocpage]{hyperref} 
\usepackage{caption} %
\usepackage{graphics,graphicx} 
\usepackage{tikz,tikz-cd} 

\usetikzlibrary{graphs,graphs.standard,calc}

\usetikzlibrary{decorations.pathreplacing,}

\usepackage{enumerate} 
\usepackage[cal=boondoxo,scr=euler]{mathalfa}

\DeclareMathAlphabet{\mathsf}{OT1}{\sfdefault}{m}{n}

\newcommand{\nocontentsline}[3]{}
\newcommand{\tocless}[2]{\bgroup\let\addcontentsline=\nocontentsline#1{#2}\egroup}

\usepackage[margin=1.25in]{geometry}
\usepackage{verbatim}

\usepackage{scalerel}

\makeatletter
\@namedef{subjclassname@2020}{\textup{2020} Mathematics Subject Classification}
\makeatother

\DeclareMathAlphabet{\amathbb}{U}{bbold}{m}{n}

\hypersetup{
    colorlinks = true,
    linkbordercolor = {white},
    linkcolor = {BrickRed},
    anchorcolor = {black},
    citecolor = {NavyBlue},
    filecolor = {cyan},
    menucolor = {red},
    runcolor = {cyan},
    urlcolor = {black}
}

\usetikzlibrary{automata}

\newtheoremstyle{teoremas}
{12pt}
{11pt}
{\itshape}
{}
{\bfseries}
{}
{.5em}
{}

\theoremstyle{teoremas}
\newtheorem{theorem}{Theorem}[section]
\newtheorem{corollary}[theorem]{Corollary}
\newtheorem{lemma}[theorem]{Lemma}
\newtheorem{proposition}[theorem]{Proposition}

\newtheorem{innercustomthm}{Theorem}
\newenvironment{mythm}[1]
  {\renewcommand\theinnercustomthm{#1}\innercustomthm}
  {\endinnercustomthm}

\newtheoremstyle{definition}
{11pt}
{11pt}
{}
{}
{\bfseries}
{}
{.5em}
{}

\theoremstyle{definition}
\newtheorem{definition}[theorem]{Definition}

\newtheorem{problem}[theorem]{Problem}
\newtheorem{question}[theorem]{Question}
\newtheorem{example}[theorem]{Example}
\newtheorem{remark}[theorem]{Remark}

\DeclareMathOperator{\rk}{rk}

\DeclareMathOperator{\cusp}{cusp}
\DeclareMathOperator{\Rel}{Rel}

\DeclareMathOperator{\trunc}{Trunc}

\newcommand{\M}{\mathsf{M}}

\newcommand{\Val}{\operatorname{Val}}

\newcommand{\N}{\mathsf{N}}
\newcommand{\U}{\mathsf{U}}

\newcommand{\T}{\mathsf{T}}

\newcommand{\LL}{\mathsf{\Lambda}}

\newcommand{\rank}{\operatorname{rk}}
\newcommand{\cl}{\operatorname{cl}}

\AtBeginDocument{%
   \def\MR#1{}
}

\title[Tutte polynomials as universal valuative invariants]
{Tutte polynomials of matroids\\ as universal valuative invariants}

\author[L.~Ferroni]{Luis Ferroni}
\address{
  School of Mathematics, Institute for Advanced Study, Princeton (NJ), United States.
}
\email{ferroni@ias.edu}

\author[B.~Schr\"oter]{Benjamin Schr\"oter}

\address{
  Department of Mathematics, KTH Royal Institute of Technology, Stockholm, Sweden.
}
\email{schrot@kth.se}

\thanks{The first author is a member of the Institute for Advanced Study supported by the Minerva Research Foundation; he was also partially supported by the Swedish Research Council grant 2018-03968. 
The second author is supported by the Swedish
Research Council grant 2022-04224.}



\allowdisplaybreaks

\begin{document}

\begin{abstract}
    We provide a full classification of all families of matroids that are closed under duality and minors, and for which the Tutte polynomial is a universal valuative invariant. There are four inclusion-wise maximal families, two of which are the class of elementary split matroids and the class of graphic Schubert matroids. As a consequence of our framework, we derive new relations among Tutte polynomials of matroids. For example, we show that the Tutte polynomial of every matroid can be expressed uniquely as an integral combination of Tutte polynomials of graphic Schubert matroids.
\end{abstract}

\maketitle

\section{Introduction}
\noindent A key feature of matroids is that there are many cryptomorphic ways to define them. When viewed under the lens of polyhedral geometry, a matroid $\M$ on the ground set $[n]:=\{1,\ldots,n\}$ is encoded through its base polytope $\mathscr{P}(\M)\subseteq \mathbb{R}^n$. In particular, one may interpret functions defined on matroids as functions on these polytopes. Under this geometric perspective, one of the most natural class of matroid functions are \emph{valuative invariants}. These are the functions on all matroids that are defined by two properties: i) they take the same value when evaluated at isomorphic matroids, and ii) roughly speaking, they behave like a measure when interpreted at the level of base polytopes. 

It is well known that not all matroid invariants are valuative. However, a remarkable fact is that many interesting invariants indeed are. The spaces of matroid valuative invariants enjoy a rich structure, and have been studied very thoroughly in the past decade, \cite{derksen,derksen-fink,bonin-kung,ardila-sanchez,eur-huh-larson,ferroni-schroter}.

In this paper we demonstrate a new connection between valuative matroid invariants and families of matroids that span subgroups of the valuative group (see \cite{eur-huh-larson} for the definition and further details). We focus on the Tutte polynomial which is the most pervasive among matroid invariants. An attractive property that this invariant features is that it is universal for deletion-contraction invariants of matroids and graphs; we refer to Brylawski and Oxley \cite{brylawski-oxley} for a more precise statement, and \cite{handbook-Tutte} for a recent collection of expository texts about Tutte polynomials. On the other hand, a result of Ardila, Fink, and Rinc\'on~\cite{ardila-fink-rincon}, proved also independently by Speyer~\cite{speyer-conjecture}, is that the Tutte polynomial is indeed a valuative invariant.

Another important valuative invariant is the so-called $\mathcal{G}$-invariant introduced by Derksen \cite{derksen} which is 
\emph{universal}, i.e., any other valuative invariant of matroids is a \emph{linear specialization} of it, e.g., we refer to \cite[Corollary~2.7]{derksen} or \cite[Theorem~1.1]{bonin-kung} to see how one derives the Tutte polynomial from the $\mathcal{G}$-invariant.
In light of this property, it is natural to believe that the $\mathcal{G}$-invariant must carry a considerable amount of information about the original matroid, and hence it should be difficult to compute. This interpretation is of course correct in some sense. In particular, one sees that the $\mathcal{G}$-invariant contains more matroid data than the Tutte polynomial. 

It is sensible to ask whether one may perhaps restrict the class of matroids, attempting that within this restricted class all valuative invariants become linear specializations of the Tutte polynomial. Without imposing some additional restrictions, it is straightforward to see that the answer is affirmative. For example, if we restrict our attention, e.g., only to the class of all uniform matroids, the Tutte polynomial serves to recover the evaluations of any valuative invariant (in fact, the full matroid, as uniform matroids are Tutte unique). More precisely, every valuative invariant on a uniform matroid may be obtained as a \emph{linear} specialization of the Tutte polynomial. 

We are therefore led to ask for \emph{maximal classes} of matroids for which the Tutte polynomial is valuatively universal. Since the Tutte polynomial satisfies a deletion-contraction recursion, and interchanges variables when taking duals, we find reasonable to restrict our attention to classes of matroids that are closed under minors and duals. 

Also further below in this article we make a precise definition of the notion of being ``valuatively universal'' for a restricted class of matroids. Our main result is the following.

\begin{restatable}{lettertheorem}{mainresult}\label{thm:main-characterization}
    Let $\mathcal{C}$ be an inclusion-wise maximal family of matroids closed under taking duals and minors, and for which the Tutte polynomial is valuatively universal. Then, one of the following four cases applies:
    \begin{enumerate}[\normalfont (a)]
        \item \label{it:class_split} $\mathcal{C}$ is the class of elementary split matroids.
        \item \label{it:class_uniform} $\mathcal{C}$ is the class of uniform matroids with additional loops and coloops.
        \item \label{it:class_minimal} $\mathcal{C}$ is the class of graphic Schubert matroids.
        \item \label{it:class_sparse} $\mathcal{C}$ is the class that contains all sparse paving matroids and
        the partition matroids listed in Lemma~\ref{lem:class_N}.
    \end{enumerate}
\end{restatable}

In other words, any class of matroids that is closed under minors and duals and for which the Tutte polynomial is valuatively universal is necessarily a subset of at least one of the four families described above. We also discuss in Section~\ref{sec:open_problems} the difficulties that may arise if one attempts to remove the assumption on the classes $\mathcal{C}$ being closed under duality.

Elementary split matroids were introduced recently, but have attracted considerable attention from the community. They can be characterized via different approaches. For instance, they can be defined as the matroids induced by compatible systems of hyperplanes splitting a hypersimplex \cite{joswig-schroter}\footnote{To be precise, Joswig and Schr\"oter introduced a class of matroids called \emph{split matroids} via a polyhedral geometry approach. The subclass of \emph{elementary split matroids} was later introduced by B\'erczi et al.~\cite{berczi}. Both classes are very similar, and in fact they agree whenever the matroids are connected.}. Alternative definitions involve hypergraphs \cite{berczi} or stressed subset relaxations \cite{ferroni-schroter}. 
One notable feature of Theorem~\ref{thm:main-characterization} is that it leads to an intrinsic characterization of elementary split matroids in a new and unexpected way.

Although the four classes of matroids listed in the preceding theorem do not compare to one another, after one restricts to connected matroids, the class of elementary split matroids strictly contains the other three. This is one reason why we regard the first of the four classes as more interesting than the other three. 

The following statement is part of Theorem~\ref{thm:main-characterization}, however we want to emphasize it, because we need it in the proof of Theorem~\ref{thm:main-characterization}.

\begin{restatable}{lettertheorem}{thmsplituniversal}\label{thm:split-universal}
    The Tutte polynomial is a universal valuative invariant on the class of elementary split matroids.
\end{restatable}

The above statement provides one possible explanation of why it is often challenging to construct two Tutte equivalent matroids $\M$ and $\N$ having the property that another given invariant $f$ distinguishes them. Since the class of elementary split matroids is expected to be very large, a na\"ive brute-force search is likely to produce matroids $\M$ and $\N$ belonging to this class. The last theorem shows that if $f$ is valuative (as it is often the case), then $f$ is a linear specialization of the Tutte polynomial and therefore one would obtain $f(\M)=f(\N)$.

A result by Kung in \cite{kung-syzygies} shows that the dimension of the linear span of the Tutte polynomials of all rank $k$ matroids on $n$ elements is $k(n-k)+1$. We  provide an independent proof of this fact.  As a combination of this with our Theorem~\ref{thm:main-characterization} we prove the following result.

\begin{restatable}{lettertheorem}{cuspidalbasis}\label{thm:cuspidals-basis-all}
    The Tutte polynomial of every matroid can be written uniquely as an integer combination of Tutte polynomials of cuspidal matroids of the same rank and size.
\end{restatable}

In this paper, a matroid is said to be \emph{cuspidal} if it is simultaneously elementary split and a Schubert matroid. In our previous work \cite{ferroni-schroter}, we proved that the evaluation of \emph{any} valuative invariant on an elementary split matroid can be written uniquely as an integer combination of the evaluations of the invariant at cuspidal matroids. However, Theorem~\ref{thm:cuspidals-basis-all} says that for the specific case of the Tutte polynomial, this property transfers to all matroids. 

It is reasonable to ask whether the remaining classes in Theorem~\ref{thm:main-characterization} lead to a variation of the statment in Theorem~\ref{thm:cuspidals-basis-all}. Indeed, the class in Theorem~\ref{thm:main-characterization}(\ref{it:class_uniform}) can be used to show that in the above statement one may replace cuspidal matroids by all uniform matroids with additional loops and coloops. 

\begin{restatable}{lettertheorem}{uniformbasis}\label{thm:uniform-loops-coloops-basis-all}
    The Tutte polynomial of every matroid can be written uniquely as an integer combination of Tutte polynomials of matroids of the same size and rank, all of which are isomorphic to uniform matroids with additional loops and coloops.
\end{restatable}

We mention that Theorem \ref{thm:cuspidals-basis-all} and Theorem \ref{thm:uniform-loops-coloops-basis-all} about Tutte polynomials of cuspidal matroids or uniform matroids with additional loops and coloops were discovered recently also by Kung in \cite[Section~5]{kung-syzygies}. Kung's work is of high relevance to our paper, though we provide an independent and self-contained approach. One of the benefits of our procedure is that it led us to discover a third and new basis for Tutte polynomials seen as the elements in a $\mathbb{Z}$-module. This new basis, described in the next theorem, provides natural answers to some interesting questions raised by Kung in \cite[Section~9]{kung-syzygies}.

\begin{restatable}{lettertheorem}{graphicschubertbasis}\label{thm:graphicschubert-basis-all}
    The Tutte polynomial of every matroid can be written uniquely as an integer combination of Tutte polynomials of graphic Schubert matroids of the same rank and size.
\end{restatable}

This is the analog of Theorem~\ref{thm:cuspidals-basis-all} but applied to the family described in Theorem~\ref{thm:main-characterization}(\ref{it:class_minimal}). Furthermore, we explain why there is no analog for the class of matroids described in Theorem~\ref{thm:main-characterization}(\ref{it:class_sparse}).

\section{Preliminaries}
\noindent 
Throughout this paper we shall assume that the reader has acquaintance with the basic tools in matroid theory, including Tutte polynomials and matroid polytopes. Our notation and conventions on matroid theory follow Oxley \cite{oxley}, and for a thorough reference on Tutte polynomials we refer to the article of Brylawski and Oxley \cite{brylawski-oxley}. For a treatment on valuative invariants with 
a focus on the $\mathcal{G}$-invariant, we refer to the book chapter by Falk and Kung in \cite{handbook-Tutte}.

We find it convenient to set up the results of this paper as a continuation of our previous manuscript \cite{ferroni-schroter}. To facilitate the navigation, we summarize in this section the definitions and results that are required for the present article.

\subsection{Elementary split and cuspidal matroids} 

Let $\M$ be a matroid on the ground set $E=[n]$ and having rank $k$. A subset $A\subseteq E$ is said to be \emph{stressed} if the restriction $\M|_A$ and the contraction $\M/A$ are both isomorphic to uniform matroids. The \emph{cusp} of a set $A\subseteq E$ is the collection of $k$-sets
\[ \operatorname{cusp}_{\M}(A) := \left\{ S\in \binom{E}{k} : |S\cap A| \geq \rk(A) + 1\right\}.\]

In \cite{ferroni-schroter} we described a generalization of the well-known circuit-hyperplane relaxation. Whenever $A$ is a stressed subset of the matroid $\M$ and $\mathscr{B}$ denotes the family of bases of $\M$, one may construct the relaxed matroid $\Rel(\M,A)$ on the ground set $E$ whose family of bases are  $\mathscr{B}\sqcup \cusp_{\M}(A)$.  
If $A$ and $A'$ are two different stressed subsets of $\M$, then $A'$ is still stressed in $\Rel(\M,A)$, and the cusp of $A'$ in $\Rel(\M,A)$ equals $\cusp_{\M}(A')$. In particular to each matroid one may associate a canonical relaxed matroid, obtained by relaxing (in any order) all the stressed subsets with non-empty cusp. The class of \emph{elementary split matroids} consists of all matroids that yield a uniform matroid when performing all the relaxations of such stressed subsets. 

A summary of equivalent characterizations of elementary split matroids is the following.
\begin{theorem}[{\cite{berczi,ferroni-schroter,ferroni-schroter-mw}}]\label{thm:elem_split_matroids}
    The following are equivalent:
    \begin{enumerate}[\normalfont(i)]
        \item $\M$ is elementary split.
        \item The matroid obtained from $\M$ by relaxing all the stressed subsets with non-empty cusp is uniform.
        \item $\M$ does not contain any minor isomorphic to $\U_{0,1}\oplus \U_{1,2}\oplus \U_{1,1}$.
        \item The proper cyclic flats of $\M$ , i.e, those cyclic flats that are neither empty or the entire ground set, form a clutter.
    \end{enumerate}
\end{theorem}

Parts of this statement are established in the work of B\'erczi, Kir\'aly, Schwarcz, Yamaguchi, and Yokoi \cite[Theorem~11]{berczi} while others appeared in \cite[Section~4]{ferroni-schroter} and \cite[Section~2.1]{ferroni-schroter-mw}. As we mentioned in the introduction, elementary split matroids are essentially equivalent to the class of split matroids, introduced by Joswig and Schr\"oter in \cite{joswig-schroter} via compatible splits of the hypersimplex. Under the assumption of connectedness, both split and elementary split matroids coincide. One reason we prefer to work with elementary split matroids rather than all split matroids is that many statements about connected elementary split matroids translate without modification to disconnected elementary split matroids, whereas for split matroids it is often required to deal with additional cases. 
On the other hand, unlike the single excluded minor established in the preceding statement, for split matroids one needs to exclude five matroids \cite{cameron-mayhew}.

A matroid is said to be a \emph{Schubert matroid} or \emph{nested matroid} if its lattice of cyclic flats is a chain. For more details about lattices of cyclic flats of matroids we refer to the article by Bonin and De Mier \cite{bonin-demier-cyclic}. 
A \emph{cuspidal matroid} is a matroid that is at the same time Schubert and elementary split. 
The prototypical cuspidal matroid is constructed as follows: within the matroid $\M=\U_{k-r,n-h}\oplus \U_{r,h}$, the subset corresponding to the ground set of the second summand is obviously stressed; the matroid obtained by relaxing it is denoted by $\LL_{r,k,h,n}$. That is, the matroid $\LL_{r,k,h,n}$ has bases 
\[ 
    \left\{ B\in \binom{[n]}{k} : |B\cap \{1,\ldots,h\}| \geq r\right\},
\]
and thus is a \emph{lattice path matroid} \cite{bonin-demier-structural} with restricted upper path and trivial lower path; see Figure~\ref{fig:cuspidal}.
Furthermore, every cuspidal matroid of rank $k$ on $n$ elements is isomorphic to some matroid $\LL_{r,k,h,n}$ with $0 \leq r\leq h$ and $0<k-r\leq n-h$.

    \begin{figure}[ht]
        \centering
        \begin{tikzpicture}[scale=0.40, line width=.4pt]

        \draw[fill=gray, fill opacity =0.4] (0,0) -- (0,3) -- (5,3) -- (5,7) -- (9,7) -- (9,0) -- cycle;
        
        \draw[line width=1.5pt,line cap=round] (0,0)--(0,3) -- (5,3)--(5,7)--(9,7);
        \draw (0,0) grid (9,7);
        \draw[decoration={brace,raise=7pt},decorate]
            (0,0) -- node[left=7pt] {$k-r$} (0,3);
        \draw (5,0) grid (9,7);
        \draw[decoration={brace,mirror, raise=4pt},decorate]
        (0,0) -- node[below=7pt] {$n-k$} (9,0);
        
        \draw[decoration={brace, raise=5pt},decorate]
        (5,7) -- node[above=7pt] {$h-r$} (9,7); 
        
        \draw[decoration={brace, raise=5pt},decorate]
        (9,7) -- node[right=7pt] {$k$} (9,0); 
        \end{tikzpicture}
        \caption{The lattice path representation of the matroid $\LL_{4,7,8,16}$.}
        \label{fig:cuspidal}
    \end{figure}
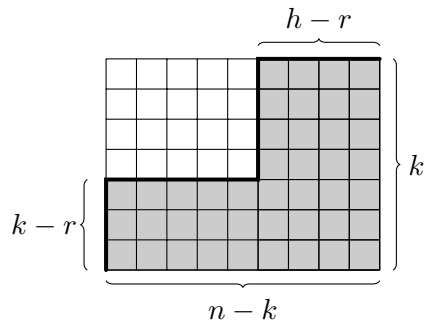

A special type of cuspidal matroid are the so-called \emph{minimal matroids} (named this way in \cite{dinolt}), which are denoted by $\T_{k,n}\cong \LL_{k-1,k,k,n}$. Their name stems from the fact that among all connected matroids on $n$ elements and having rank $k$, the matroid $\mathsf{T}_{k,n}$ is the only one attaining the minimal possible number of bases, $k(n-k)+1$. Minimal matroids are graphic, and hence binary. The minimal matroid $\mathsf{T}_{k,n}$ corresponds to a cycle of length $k+1$ in which one edge is replaced by $n-k$ parallel copies, see Figure~\ref{fig:minimal}.

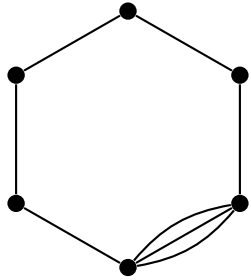
\begin{figure}[ht]
        \centering
        \begin{tikzpicture}  
    	[scale=0.05,auto=center,every node/.style={circle, fill=black, inner sep=2.3pt},every edge/.append style = {thick}] 
    	\tikzstyle{edges} = [thick];
        \graph  [empty nodes, clockwise, radius=1em,
        n=9, p=0.3] 
            { subgraph C_n [n=6,clockwise,radius=1.7cm,name=A]};
            
		\draw[edges] (A 3) edge[bend right=20] (A 4);
		\draw[edges] (A 3) edge[bend right=-20] (A 4);
    \end{tikzpicture}\caption{The underlying graph of a minimal matroid $\T_{5,8}$}\label{fig:minimal}
\end{figure}

\subsection{Valuative invariants and universality}

Throughout this paper we will denote by $\mathcal{M}_{E}$ the set of all matroids on ground set $E$. The notation $\mathcal{M}_{E,k}$ stands for the matroids on $E$ having rank $k$. Whenever $E=[n]$ we often write $\mathcal{M}_n$ or $\mathcal{M}_{n,k}$. The disjoint union of all $\mathcal{M}_{n,k}$ for every pair of integers $k$ and $n$ such that $0\leq k\leq n$ will be denoted by $\mathcal{M}$.
 
Recall that every matroid $\M$ has associated to it its base polytope $\mathscr{P}(\M)\subseteq \mathbb{R}^E$ which is the convex hull of all the indicator vectors of its bases. The \emph{indicator function} of this polytope is the map $\amathbb{1}_{\mathscr{P}(\M)}:\mathbb{R}^E\to \mathbb{Z}$ defined by
    \[ \amathbb{1}_{\mathscr{P}(\M)}(x) = \begin{cases} 1 & x\in \mathscr{P}(\M), \\ 0 & \text{otherwise.}\end{cases}\]

Let $A$ be any (additive) abelian group. A function $f:\mathcal{M}_{E,k}\to A$ is said to be \emph{valuative} if $\sum_{i=1}^s a_i\cdot f(\M_i) = 0$ whenever the function $x\mapsto\sum_{i=1}^s a_i\cdot  \amathbb{1}_{\mathscr{P}(\M_i)}(x)$ from $\mathbb{R}^E$ to $\mathbb{Z}$ is identically zero. A result of Derksen and Fink \cite[Corollary~3.9]{derksen-fink} shows that this is equivalent to requiring that $f$ behaves as a measure under matroid polytope subdivisions, i.e., that
    \[
    f(\M) = \sum_{I\subseteq [s]} (-1)^{|I|-1} f\Bigl(\bigcap_{i\in I} \M_i\Bigr),
    \]
whenever $\{\M_1,\ldots,\M_s\}$ is a list of matroids for which the collection of their base polytopes $\left\{\mathscr{P}(\M_1),\ldots\mathscr{P}(\M_s)\right\}$ forms the maximal cells of a subdivision of $\mathscr{P}(\M)$, and $\bigcap_{i\in I} \M_i$ denotes the matroid whose base polytope is $\bigcap_{i\in I} \mathscr{P}(\M_i)$. For more on other versions of the notion of valuation, we refer to \cite[Appendix~A]{eur-huh-larson}. 

A function $f:\mathcal{M}\to X$ with image in the set $X$ is said to be a \emph{matroid invariant} if $f(\M) = f(\N)$ whenever $\M$ and $\N$ are isomorphic matroids. If $f:\mathcal{M}\to A$ is simultaneously a valuation and a matroid invariant, we call $f$ a \emph{valuative invariant}. 

\begin{example}
    The map $T:\mathcal{M}\to \mathbb{Z}[x,y]$ that associates to each matroid $\M$ its Tutte polynomial $T(\M):=T_{\M}(x,y)$ is a valuative invariant. This is a very non-trivial feature of the Tutte polynomial. For a proof, we refer to \cite[Theorem~5.4]{ardila-fink-rincon} or \cite[Lemma~3.4]{speyer-conjecture}.
\end{example}

Since the Tutte polynomial is a valuative invariant, one may conclude that all linear specializations of the Tutte polynomial are valuative invariants as well. This comprises well-known matroid statistics such as the number of bases, the number of independent sets, or the $\beta$-invariant of the matroid. In this article we would like to understand the classes of matroids for which the Tutte polynomial is in a sense ``the most general'' valuative invariant. This leads us to formulate the following definition.

\begin{definition}\label{def:universal-valuation}
    Let $u:\mathcal{M}\to A$ be a valuative invariant, and $\mathcal{C}\subseteq \mathcal{M}$ a subclass of matroids. We say that $u$ is \emph{valuatively universal within $\mathcal{C}$} if for each valuative invariant $f:\mathcal{M}\to A'$ there exists a homomorphism $\widehat{f}:A\to A'$ such that
        \[ f(\M) = \widehat{f}(u(\M))\]
    for all matroids $\M\in \mathcal{C}$. In other words $f(\M)$ is a linear specialization of $u(\M)$. 
\end{definition}

Let us denote by $\Val(\mathcal{M},A)$ the space of valuations from $\mathcal{M}$ into the abelian group~$A$. Since the set $\mathcal{M}$ is the disjoint union of all the strata $\mathcal{M}_{n,k}$,
one obtains the direct sum decomposition
    \[ \Val(\mathcal{M}, A) = \bigoplus_{n,k} \Val(\mathcal{M}_{n,k},A).
    \]
This decomposition can be transferred to any subclass $\mathcal{C}\subseteq \mathcal{M}$ of matroids, as one may stratify $\mathcal{C}$ by the disjoint sets $\mathcal{C}_{n,k}:= \mathcal{M}_{n,k} \cap \mathcal{C}$.
Hence, describing a valuation on all of $\mathcal{M}$ is actually equivalent to describing its restrictions to each set $\mathcal{M}_{n,k}$. On the other hand, checking that an invariant $u:\mathcal{M}\to A$ is a universal valuative invariant within $\mathcal{C}$ requires constructing for each valuative invariant $f:\mathcal{M}\to A'$ an infinite family of homomorphisms $\{\widehat{f}_{n,k}\}_{0\leq k\leq n} \subseteq \operatorname{Hom}(A,A')$ such that $f(\M) = \widehat{f}_{n,k}(u(\M))$ for each $\M\in \mathcal{C}_{n,k}$.

Note that if $u:\mathcal{M}\to A$ is valuatively universal within the class of matroids $\mathcal{C}$, then $u$ is valuatively universal within any subclass $\mathcal{C}'\subseteq \mathcal{C}$.

A central result of Derksen and Fink \cite[Theorem~1.4]{derksen-fink}, previously conjectured by Derksen \cite{derksen}, is that a certain invariant $\mathcal{G}:\mathcal{M}\to \mathbf{U}$ is valuatively universal within all matroids $\mathcal{M}$. In other words, any matroid valuative invariant is a linear specialization of $\mathcal{G}$. It follows that the invariant $\mathcal{G}$ is valuatively universal within every class $\mathcal{C}$.  
This invariant plays a major role in this article, so below we provide its definition. 

We begin by describing a certain abelian group $\mathbf{U}$ as follows. Let $\mathbf{U}_n$ be the free abelian group generated by the $2^n$ symbols $U_{s}$
indexed by the $0/1$-vectors $s\in \mathbb{R}^n$. The group $\mathbf{U}$ is the direct sum of all groups $\mathbf{U}_{n}$ for $n\geq 0$. Now we are able to define the valuation $\mathcal{G}$.

\begin{definition}\label{def:ginvar}
    Let $\M$ be a matroid on $E=[n]$. Consider the set of all maximal flags of subsets of $E$ of the form $C=[\varnothing = S_0 \subsetneq S_1 \subsetneq \cdots \subsetneq S_n = E]$. There are $n!$ of these flags.
    To the flag $C$ we associate the generator $U_{s}$ of $\mathbf{U}_{n}$, indexed by the tuple $s=s(C)=(s_1,\ldots,s_n)\in\{0,1\}^n$, 
    with $s_i:=\rank_\M(S_i)-\rank_\M(S_{i-1})$ for $1 \leq i\leq n$.
    For each integer $n\geq 0$, we define the map $\mathcal{G}_{n}:\mathcal{M}_{n}\to \mathbf{U}_n$ by
        \[ \mathcal{G}_{n}(\M) = \sum_{C} U_{s(C)}\]
        where the sum is taken over all maximal flags of the ground set of $\M$.
    The $\mathcal{G}$-invariant is the map $\mathcal{G}:\mathcal{M}\to \mathbf{U}$ with $\mathcal{G}|_{\mathcal{M_n}} = \mathcal{G}_{n}$, that is, $\mathcal{G} = \bigoplus_{n\geq 0} \mathcal{G}_n$.
\end{definition}

\section{A reformulation of universality for valuations}

\noindent The purpose of this short section is to introduce the basic notions that plays an instrumental role throughout the proof of our main results. 

\begin{definition}
    Let $f:\mathcal{M}_{n,k}\to A$ be a valuative invariant, and let $\mathcal{C}_{n,k}\subseteq\mathcal{M}_{n,k}$ be a subclass of matroids on $n$ elements and rank $k$. Consider the $\mathbb{Z}$-submodule of $A$ spanned by all the elements $f(\M)$ for $\M\in \mathcal{C}_{n,k}$. The (torsion-free) rank of this $\mathbb{Z}$-module is called the \emph{$f$-rank} of $\mathcal{C}_{n,k}$.
\end{definition}

Typically, $A$ is the ring of integers, or a ring of polynomials over the integers. In these cases, $A$ is also a free abelian group.

\begin{example}\label{example:tuttes42}
    Consider $\mathcal{C}_{4,2}=\mathcal{M}_{4,2}$ all matroids on four elements and rank two. In other words, $\mathcal{C}_{4,2}$ consists of all possible relabellings of the seven isomorphism classes of the matroids $\U_{0,2}\oplus\U_{2,2}$, $\U_{0,1}\oplus\U_{1,2}\oplus\U_{1,1}$, $\U_{1,3}\oplus \U_{1,1}$, $\U_{0,1}\oplus \U_{2,3}$, $\U_{1,2}\oplus \U_{1,2}$, $\mathsf{T}_{2,4}$, and $\U_{2,4}$.  Their Tutte polynomials are
    \begin{align*}
        T(\U_{2,4}) &=x^{2} + y^{2} + 2 x + 2 y, &
        T(\mathsf{T}_{2,4}) &=x^{2} + x y + y^{2} + x + y,\\
        T(\U_{1,2}\oplus \U_{1,2}) &=x^{2} + 2 x y + y^{2}, &
        T(\U_{0,2}\oplus\U_{2,2}) &= x^{2} y^{2}\\
        T(\U_{1,3}\oplus \U_{1,1}) &=x y^{2} + x^{2} + x y, &
        T(\U_{0,1}\oplus \U_{2,3}) &=x^{2} y + x y + y^{2},\\
        T(\U_{0,1}\oplus\U_{1,2}\oplus\U_{1,1}) &= x^{2} y + x y^{2}. & &
    \end{align*}
    The $T$-rank (or ``Tutte rank'') of $\mathcal{C}_{4,2}=\mathcal{M}_{4,2}$ is therefore the rank of the $\mathbb{Z}$-span of these seven polynomials which is five (cf. Corollary~\ref{coro:T-rank-of-matroids} below). 
    This can be read off from the two relations
    \begin{align}
        2\,T(\mathsf{T}_{2,4}) &= T(\mathsf{U}_{2,4}) + T(\U_{1,2}\oplus \U_{1,2})\\
         T(\U_{2,3}\oplus \U_{0,1}) +  T(\U_{1,3}\oplus \U_{1,1}) &= T(\U_{0,1}\oplus \U_{1,2}\oplus\U_{1,1})+T(\U_{1,2}\oplus \U_{1,2}) \label{eq:relation1}
    \end{align}
    which imply also the relation
        \begin{align}\label{eq:relation2}
        2\,T(\mathsf{T}_{2,4}) + T(\U_{0,1}\oplus \U_{1,2}\oplus\U_{1,1}) &= T(\mathsf{U}_{2,4}) + T(\U_{0,1}\oplus \U_{2,3}) +  T(\U_{1,3}\oplus \U_{1,1})
    \end{align}
    and checking that the remaining polynomials are independent after eliminating $T(\mathsf{T}_{2,4})$ and $T(\U_{0,1}\oplus \U_{1,2}\oplus\U_{1,1})$.
\end{example}

\begin{example}\label{example:gs42}
    As before, consider $\mathcal{C}_{4,2}=\mathcal{M}_{4,2}$. The corresponding seven $\mathcal{G}$-invariants are
    \begin{align*}
    \mathcal{G}(\U_{0,2}\oplus\U_{2,2}) &= 4\,U_{(0, 0, 1, 1)} + 4\,U_{(0, 1, 0, 1)} + 4\,U_{(0, 1, 1, 0)} + 4\,U_{(1, 0, 0, 1)} + 4\,U_{(1, 0, 1, 0)}\\ &\quad + 4\,U_{(1, 1, 0, 0)},\\
    \mathcal{G}(\U_{0,1}\oplus\U_{1,2}\oplus\U_{1,1}) &= 2\,U_{(0, 1, 0, 1)} + 4\,U_{(0, 1, 1, 0)} + 4\,U_{(1, 0, 0, 1)} + 6\,U_{(1, 0, 1, 0)} + 8\,U_{(1, 1, 0, 0)},\\
    \mathcal{G}(\U_{1,3}\oplus \U_{1,1}) &= 6\,U_{(1, 0, 0, 1)} + 6\,U_{(1, 0, 1, 0)} + 12\,U_{(1, 1, 0, 0)},\\
    \mathcal{G}(\U_{0,1}\oplus \U_{2,3}) &= 6\,U_{(0, 1, 1, 0)} + 6\,U_{(1, 0, 1, 0)} + 12\,U_{(1, 1, 0, 0)},\\
    \mathcal{G}(\mathsf{T}_{2,4}) &= 4\,U_{(1, 0, 1, 0)} + 20\,U_{(1, 1, 0, 0)},\\
    \mathcal{G}(\U_{1,2}\oplus \U_{1,2}) &= 8\,U_{(1, 0, 1, 0)} + 16\,U_{(1, 1, 0, 0)},\\
    \mathcal{G}(\U_{2,4}) &= 24\,U_{(1, 1, 0, 0)}.
\end{align*}
    A direct computation reveals that there is a single relation among them, namely
    \begin{align*}
        2\, \mathcal{G}(\mathsf{T}_{2,4}) &= \mathcal{G}(\mathsf{U}_{2,4}) + \mathcal{G}(\U_{1,2}\oplus \U_{1,2}),
    \end{align*}
    hence the $\mathcal{G}$-rank of $\mathcal{M}_{4,2}$ is six (cf. Remark~\ref{remark:G-rank-of-matroids} below).
\end{example}

The following proposition provides a concrete way of addressing the problem of showing that a given valuative invariant is valuatively universal within a given class of matroids.

\begin{proposition}\label{prop:equivalences-universality}
    Let $\mathcal{C}$ be a class of matroids and let $u:\mathcal{M}\to A$ be a valuative invariant where $A$ is a free $\mathbb{Z}$-module. The following two conditions are equivalent.
    \begin{enumerate}[\normalfont (i)]
        \item The invariant $u$ is valuatively universal within $\mathcal{C}$.
        \item For each stratum $\mathcal{C}_{n,k}$, the $u$-rank equals the $\mathcal{G}$-rank.
    \end{enumerate}
\end{proposition}

\begin{proof}
    (i) $\Rightarrow$ (ii). Let us assume that the map $u$ is valuatively universal within $\mathcal{C}$. The Definition~\ref{def:universal-valuation} of a valuatively universal invariant applied to the valuation $f=\mathcal{G}$  reveals that there must exist a group homomorphism $\widehat{\mathcal{G}}:A\to \mathbf{U}$ such that
        \[ \mathcal{G}(\M) = \widehat{\mathcal{G}}(u(\M))\]
    for every $\M\in \mathcal{C}$. In particular the $\mathcal{G}$-rank of each stratum $\mathcal{C}_{n,k}$ is bounded from above by the $u$-rank of $\mathcal{C}_{n,k}$.
    
    On the other hand, since the $\mathcal{G}$-invariant is valuatively universal within all matroids~$\mathcal{M}$, the valuative invariant $u$ is a homomorphic image of $\mathcal{G}$, i.e., there exists a homomorphism $\widehat{u}:\mathbf{U}\to A$ such that
        \[ u(\M) = \widehat{u}(\mathcal{G}(\M))\]
    for every $\M\in\mathcal{M}$. When restricting this map to each family $\mathcal{C}_{n,k}\subseteq \mathcal{M}_{n,k}$ we see that the $u$-rank is bounded from above by the $\mathcal{G}$-rank. In particular, these two quantities must agree. 

    (ii) $\Rightarrow$ (i). 
    For the valuative invariant $u$ one may find a  homomorphism $\widehat{u}:\mathbf{U}\to A$ such that $u = \widehat{u}\circ \mathcal{G}$ as in the previous paragraph. By assumption, the $\mathbb{Z}$-submodules $\mathbf{U}_{n,k}'\subseteq \mathbf{U}$ spanned by $\mathcal{G}(\mathcal{C}_{n,k})$ and $A'\subseteq A$ spanned by $u(\mathcal{C}_{n,k})$ have the same rank.
    Since the homomorphism $\widehat{u}$ induces an epimorphism $u'_{n,k}:\mathbf{U}'_{n,k} \twoheadrightarrow A'$, and these two free $\mathbb{Z}$-modules have the same rank, by a standard result on free $\mathbb{Z}$-modules then $u'_{n,k}$ is actually an isomorphism, see for example \cite{orzech}. In particular, the inverse map $(u'_{n,k})^{-1}: A' \to \mathbf{U}'_{n,k}\subseteq \mathbf{U}$ is a well-defined homomorphism. Furthermore, it is clear that it satisfies
        \[ 
        \mathcal{G}(\M) = (u'_{n,k})^{-1}(u(\M))
        \]
    for each $\M\in\mathcal{C}_{n,k}$. 
    In other words, the $\mathcal{G}$-invariant when restricted to $\mathcal{C}_{n,k}$ factors through $u$. This holds true for every $n$ and $k$ and all matroids in $\mathcal{C}$.
    Thus we obtain the composition $\mathcal{G} = \widehat{\mathcal{G}}\circ u$ for the homomorphism $\widehat{\mathcal{G}}: A\to \mathbf{U}$ which is the direct sum of all the homomorphisms $(u'_{n,k})^{-1}$. On the other hand, every valuative invariant $f:\mathcal{M}\to X$ factors through $\mathcal{G}$. Thus, it is immediate from the above that $u$ is valuatively universal within $\mathcal{C}$, as required.
\end{proof}

Before concluding this section, let us state and prove a technical lemma that we will require in the proofs of Theorems~\ref{thm:cuspidals-basis-all}, \ref{thm:uniform-loops-coloops-basis-all} and \ref{thm:graphicschubert-basis-all}.

\begin{lemma}\label{lemma:technical-lemma}
    Let $\mathcal{C}_{n,k}$ be a class of matroids, and let $\M_1,\ldots,\M_s\in \mathcal{C}_{n,k}$ be non-isomorphic Schubert matroids. Let $u:\mathcal{M}_{n,k}\to \mathbb{Z}$ be a valuative invariant that is universal within $\mathcal{C}_{n,k}$. Assume that for every matroid $\M\in \mathcal{M}_{n,k}$ there are rational numbers $\mu_1,\ldots,\mu_s$ (depending on $\M)$ such that:
        \[ u(\M) = \sum_{i=1}^s \mu_i\cdot u(\M_i).\]
    Then for every $\M$ the numbers $\mu_i$, $1\leq i\leq s$, are in fact integers.
\end{lemma}

\begin{proof}
    Let us call $\overline{\mathcal{S}}_{n,k}$ the set of all isomorphism classes of Schubert matroids on $n$ elements and rank $k$. For each $1\leq i\leq s$ consider the function $f_i :\overline{\mathcal{S}}_{n,k}\to\mathbb{Z}$ defined by
        \[ f_i([\M]) = \begin{cases}
            1 & \text{if $[\M]=[\M_i]$}\\ 0 & \text{otherwise},
        \end{cases}\]
    where the brackets denote isomorphism classes. A result of Derksen and Fink \cite[Theorem~6.3]{derksen-fink} guarantees that each function $f_{i}$ can be extended uniquely to a valuative invariant $F_{i}:\mathcal{M}_{n,k}\to \mathbb{Z}$. When restricted to the class $\mathcal{C}_{n,k}$, each valuative invariant $F_{i}$ becomes a specialization of the invariant $u$. In other words, there are homomorphisms $\widehat{F}_{i}:\mathbb{Z}\to \mathbb{Z}$ with $F_{i}(\M) =\widehat{F}_{i}\circ u(\M)$ for every $\M\in\mathcal{C}_{n,k}$. Now, let us denote by $\phi_{i}$ the function $\widehat{F}_{i}\circ u:\mathcal{M}_{n,k}\to\mathbb{Z}$ which agrees with $F_{i}$ for matroids in $\mathcal{C}_{n,k}$ but may differ for other matroids.

    By assumption, for every matroid $\M\in\mathcal{M}_{n,r}$, we can find rational numbers $\mu_{i}$ such that
    \[ u(\M) = \sum_{i=1}^s \mu_{i}\cdot u(\M_i).\]
    By applying the homomorphism $\widehat{F}_{j}$ on both sides, we obtain:
    \[ 
    \phi_{j}(\M) = \sum_{i=1}^s\mu_{i}\cdot F_{j}(\M_i) = \sum_{i=1}^s\mu_{i}\cdot f_{j}([\M_i]) = \mu_{j}.
    \]
    In particular, we deduce that each rational coefficient $\mu_{j}$ must in fact be integral, as they are in the range of the integer-valued function $\phi_{j}$.
\end{proof}

\section{Four families for which Tutte polynomials are  universal }

\noindent In this section we describe four families of matroids for which the Tutte polynomial is a universal valuative invariant. These four families appear in the statement of our main result, Theorem~\ref{thm:main-characterization}.
Beyond the first of these four, which is the class of elementary split matroids, the remaining three classes might appear somewhat unmotivated until the reader reaches the proof of the main theorem in Section~\ref{sec:main-proof}.

In this section we prove furthermore some results that apply for relevant families of matroids beyond the aforementioned four. For instance, we calculate the $T$-rank of $\mathcal{M}_{n,k}$ explicitly, thus recovering a result by Kung \cite{kung-syzygies}. Furthermore, we answer questions posed by Kung by calculating the $T$-rank of the classes of binary and even graphic matroids.  

\subsection{Elementary split matroids and Tutte universality}

In this subsection we show that within the class $\mathcal{C}$ of all elementary split matroids the Tutte polynomial is a universal valuative invariant. To this end, the following is the first step.

\begin{proposition}\label{prop:tutte-cuspidals-independent}
    There are exactly $k(n-k)+1$ non-isomorphic cuspidal matroids on $n$ elements and rank $k$.
    Moreover, their Tutte polynomials are independent in $\mathbb{Z}[x,y]$.
\end{proposition}

\begin{proof}
    Fix parameters $k$ and $n$. We have to consider only all cuspidal matroids $\LL_{r,k,h,n}$ with $0\leq r \leq h$ and $0 \leq k-r \leq n-h$. Note that the cuspidal matroid $\LL_{r,k,h,n}$ is isomorphic to the uniform matroid $\U_{k,n}$ whenever $r=0$ or $h-r=n-k$.
    Thus we are left with the $k(n-k)$ matroids $\LL_{r,k,h,n}$ where $0\leq k-r \leq k-1$ and $0\leq h-r \leq n-k-1$ as well as the uniform matroid $\U_{k,n}$.
    Next we show that their $k(n-k)+1$ Tutte polynomials are linearly independent and thus different which implies that these matroids are indeed non-isomorphic.

    Instead of the Tutte polynomials consider the polynomials $P_{r,k,h,n} (x,y) := T_{\LL_{r,k,h,n}}(x+1,y+1)$ for all $r$ and $h$ with $0\leq k-r \leq k-1$ and $0\leq h-r \leq n-k-1$. 
    It follows from the formula of the rank function of $\LL_{r,k,h,n}$ in \cite[Corollary~3.25]{ferroni-schroter} that
    \begin{align*}
        P_{r,k,h,n}(x,y) =& \sum_{A\subseteq E} x^{k-\rank A} y^{|A|-\rank A}\\
        =& \sum_{j=0}^r \binom{h}{j} \left( \sum_{i=0}^{m-1} \binom{n-h}{i} x^{r-j} y^{m-i} 
        + \sum_{i=0}^{k-r} \binom{n-h}{i} x^{k-i-j}\right)\\
        &+ \sum_{j=r+1}^h \binom{h}{j} \left( \sum_{i=0}^{k-j} \binom{n-h}{i} x^{k-i-j} 
        + \sum_{i=k-j+1}^{n-h} \binom{n-h}{i} y^{i+j-k}\right)
    \end{align*}
    where $m = n-k+r-h$.
    We see that the monomial $x^r y^{n-k+r-h}$ has coefficient~$1$ and that this monomial has maximal total degree among all monomials that are divisible by~$xy$. Moreover, 
    the Tutte polynomial $T_{\U_{k,n}}(x+1,y+1)$ has no monomial with a factor~$xy$.
    Therefore the $k(n-k)$ polynomials $P_{r,k,h,n}(x,y)$ together with $T_{\U_{k,n}}(x+1,y+1)$ are independent.
    The constant shift of the variables is an invertible linear transformation on the space of polynomials $\mathbb{Z}[x,y]$, hence we conclude that the Tutte polynomials of all non-isomorphic cuspidal matroids are independent which completes the proof.
\end{proof}

Now, we have all the ingredients to prove Theorem~\ref{thm:split-universal}.

\thmsplituniversal*

\begin{proof}
    In this prove we make use of Proposition~\ref{prop:equivalences-universality}. Let us denote by $\mathcal{C}$ the class of elementary split matroids for the moment. To show that the Tutte polynomial is universal we shall prove that the $T$-rank and the $\mathcal{G}$-rank of $\mathcal{C}_{n,k}$ are equal. In fact, we will show that they are both equal to $k(n-k)+1$. To this end, we rely on \cite[Theorem~6.6]{ferroni-schroter}. The $\mathcal{G}$-invariant of any elementary split matroid $\M\in \mathcal{C}_{n,k}$ can be calculated via:
    \begin{equation} \label{eq:g-split}
    \mathcal{G}(\M) = \mathcal{G}(\U_{k,n}) - \sum_{r,h} \lambda_{r,h} \left( \mathcal{G}(\LL_{r,k,h,n}) - \mathcal{G}(\U_{k-r,n-h}\oplus \U_{r,h})\right),
    \end{equation}
    where $\lambda_{r,h}$ records the number of stressed subsets of rank $r$, size $h$ and non-empty cusp of the matroid $\M$. In fact, by \cite[Example~6.5]{ferroni-schroter}, we have that $\mathcal{G}(\U_{k-r,n-h}\oplus \U_{r,h}) = \mathcal{G}(\LL_{k-r,k,n-h,n}) + \mathcal{G}(\LL_{r,k,h,n}) - \mathcal{G}(\U_{k,n})$. In particular, after substituting the last identity in equation~\eqref{eq:g-split}, all the matroids appearing on the right hand side are cuspidal. It follows that the $\mathcal{G}$-invariant of an elementary split matroid can be written as an integer combination of $\mathcal{G}$-invariants of cuspidal matroids. Since there are $k(n-k)+1$ non-isomorphic cuspidal matroids on $n$ elements and rank $k$, we have that the $\mathcal{G}$-rank of $\mathcal{C}_{n,k}$ is at most $k(n-k)+1$.

    On the other hand, since the map $T$ is a linear specialization of $\mathcal{G}$, the $T$-rank of $\mathcal{C}_{n,k}$ is bounded from above by the $\mathcal{G}$-rank. However, from Proposition~\ref{prop:tutte-cuspidals-independent} we know that the $T$-rank of $\mathcal{C}_{n,k}$ is at least $k(n-k)+1$, because all cuspidal matroids are elementary split and their Tutte polynomials are independent. 

    From the two paragraphs above, it follows that the $\mathcal{G}$-rank and the $T$-rank of $\mathcal{C}_{n,k}$ must be equal to $k(n-k)+1$, so the proof is complete.
\end{proof}

Reasoning as in the proof of the last statement, it is evident that the Tutte polynomial of any elementary split matroid can be written as an integer combination of Tutte polynomials of cuspidal matroids. The next result, which can be obtained as an application of the preceding theorem, shows that in fact the Tutte polynomial of \text{any} matroid can be written uniquely as a combination of Tutte polynomials of cuspidal matroids.

\cuspidalbasis*

\begin{proof}
    To prove this statement, let us show first that the $T$-rank of $\mathcal{M}_{n,k}$ is $k(n-k)+1$. The fact that the $T$-rank is at least $k(n-k)+1$ follows immediately from 
    Proposition~\ref{prop:tutte-cuspidals-independent}. To prove this is also the upper bound, note that all Tutte polynomials of matroids on $n$ elements and rank $k$ are of the form:
        \[ T_\M(x,y) = \sum_{i=0}^{k}\sum_{j=0}^{n-k} a_{ij} x^i y^j\]
    for some coefficients $a_{ij}\in\mathbb{Z}$.
    In particular, the rank of the integer span of all Tutte polynomials of matroids in $\mathcal{M}_{n,k}$ is at most $(k+1)(n-k+1)$.
    However, it is well-known that the coefficients of any Tutte polynomial of a matroid on $n$ elements and rank $k$ satisfies the $n$ linear independent relations 
    \[
        \sum_{i=0}^s \sum_{j=0}^{s-i} (-1)^j \binom{s-i}{j} a_{ij} = 0 
    \]
    for $0\leq s < n$;
    see \cite[Theorem~6.2.13]{brylawski-oxley}. In particular, the $T$-rank of $\mathcal{M}_{n,k}$ is at most $(k+1)(n-k+1) - n = k(n-k) + 1$, as we claimed. 
    
    We have two free $\mathbb{Z}$-modules, the spans of $T(\mathcal{C}_{n,k})$ and $T(\mathcal{M}_{n,k})$ both of which have the same rank, while the former is a submodule of the latter. In general, this allows to conclude that the Tutte polynomial of any matroid $\M\in \mathcal{M}_{n,k}$ can be written uniquely as an \emph{a priori rational} combination of Tutte polynomials of cuspidal matroids (see, e.g., \cite[Theorem~II.1.6]{hungerford}). The integrality follows automatically by Lemma~\ref{lemma:technical-lemma}, because cuspidal matroids are Schubert and elementary split, and $T$ is valuatively universal within the class of elementary split matroids by Theorem~\ref{thm:split-universal}.
\end{proof}

The above statement (more precisely, the first part of its proof) implies the following corollary that we will require several times below.

\begin{corollary}\label{coro:T-rank-of-matroids}
    The $T$-rank of $\mathcal{M}_{n,k}$ equals $k(n-k)+1$.
\end{corollary}

To the best of our knowledge, the last corollary was first recorded in the work of Kung \cite{kung-syzygies}, although it can be deduced from Brylawski's main result \cite{brylawski-decomposition}. In either Kung's framework or ours, to produce a proof of Theorem~\ref{thm:cuspidals-basis-all}, one needs to rely on a version of Proposition~\ref{prop:tutte-cuspidals-independent}. We refer to \cite[Section~8]{kung-syzygies} for an alternative proof of that lemma.

\begin{remark}\label{remark:G-rank-of-matroids}
    The $\mathcal{G}$-rank of $\mathcal{M}_{n,k}$ is known to be $\binom{n}{k}$; this is equivalent to the content of \cite[Theorem~1.5(a)]{derksen-fink}.
\end{remark}

\subsection{A second class of matroids containing all sparse paving matroids}\label{sec:the_class_N}

Now we focus our attention on a different class of matroids, also appearing within the statement of Theorem~\ref{thm:main-characterization}.

Let $\mathcal{N}$ denote the class of matroids that have no minor isomorphic to $\U_{1,1}\oplus\U_{1,3}$ nor $\U_{0,1}\oplus\U_{2,3}$.

\begin{lemma}\label{lem:class_N} 
    The class $\mathcal{N}$is closed under duality, it contains all sparse paving matroids, and also the following disconnected partition matroids on $n$ elements\footnote{To avoid ambiguities, we point out that in this article a \emph{partition matroid} is a matroid that is a direct sum of uniform matroids. This is more general than the definition given in \cite{oxley}, but agrees with the usage of the term, for example, in \cite{Lawler:1976}.}:
    \begin{itemize}
        \item $\U_{0,n-k-\ell}\oplus(\U_{1,2})^\ell \oplus \U_{k-\ell,k-\ell}$ for all ranks $0\leq k\leq n$ and $0\leq \ell\leq \min\{k,n-k\}$,
        \item $\U_{0,n-\ell}\oplus \U_{1,\ell}$ of rank $1$ for all $3 \leq \ell \leq n-1$
        \item $\U_{\ell-1,\ell}\oplus \U_{n-\ell,n-\ell}$ of rank $n-1$ for all $3 \leq \ell \leq n-1$
        
    \end{itemize}
\end{lemma}

\begin{proof}
    The two excluded minors are duals of one another, thus the class $\mathcal{N}$ is closed under duality. A matroid is sparse paving if it has no minor isomorphic to $\U_{0,1}\oplus\U_{2,2}$ or $\U_{0,2}\oplus\U_{1,1}$. Each of these two matroids is a proper minor of one of the two excluded minors of the class $\mathcal{N}$, hence $\mathcal{N}$ contains all sparse paving matroids.
    It is straightforward to see that the listed partition matroids avoid both excluded minors that define the class~$\mathcal{N}$.
\end{proof}

The preceding statement shows that all matroids having rank $0$ and $1$, as well as their duals, are encompassed by the class $\mathcal{N}$.

\begin{lemma} \label{lem:class_N2}
    Let $\M$ be a matroid lying in the class $\mathcal{N}$. Then one of the following two cases applies:
    \begin{enumerate}[\normalfont (i)]
        \item $\M$ is connected and sparse paving.
        \item $\M$ is disconnected and coincides with a partition matroid listed in Lemma~\ref{lem:class_N}.
    \end{enumerate}
\end{lemma}

\begin{proof}
    Assume $\M$ is connected, has rank $k$ and is not paving.
    Then there is a circuit of $\M$ with $|C|\leq k-1$. Choose $e\in C$ and complete the independent set $C\setminus\{e\}$ to a basis~$B$. By assumption there are two distinct elements $f, g\in B\setminus C$. Moreover, because $\M$ is connected there must be an element $h\not\in B$ such that $h\not\in \cl(C\cup\{f\})$ and $h\not\in \cl(C\cup\{g\})$ otherwise $f$ or $g$ would be a coloop.
    Now consider the contraction $\M/(B\setminus\{f,g\})$. This matroid has the loop $e$ and the circuit $\{f,g,h\}$ which certifies that $\U_{0,1}\oplus\U_{2,3}$ is a minor.
    By duality we conclude that connected matroids in $\mathcal{N}$ are sparse paving.

    For the second part of the statement note that there are only two connected matroids on three elements, namely the uniform matroids $\U_{1,3}$ and $\U_{2,3}$.
    Thus every connected component of a matroid $\M\in\mathcal{N}$ with at least three elements has one of those as a minor. Furthermore, if both the rank and corank of $\M$ are at least two then both are minors, which cannot happen as a second connected component would lead to an additional loop or coloop.
    Thus every connected component of a disconnected matroid in $\mathcal{N}$ has rank $0$,$1$, corank $0$ or corank $1$. In other words they are $\U_{0,1}$, $\U_{1,\ell}$, $\U_{1,1}$ or $\U_{\ell-1,\ell}$ for some $\ell$. The feasible combinations of those matroids are listed in Lemma~\ref{lem:class_N}.
\end{proof}

\begin{proposition} \label{prop:class_N_tutte}
    Fix $k$ and $n$ with $2\leq k \leq n-2$.
    The Tutte polynomials of the disconnected matroids in $\mathcal{N}_{n,k}$ together with $T(\U_{k,n})$ and $T(\LL_{1,k,n-k,n})$ form a basis for $T(\mathcal{N}_{n,k})$.
    In particular, whenever $2\leq k\leq n-2$, the $T$-rank of $\mathcal{N}_{n,k}$ is $\min\{k+3,n-k+3\}$.
\end{proposition}

\begin{proof}
    By Lemma~\ref{lem:class_N2}, the disconnected matroids in $\mathcal{N}_{n,k}$ are of the form $\U_{0,n-k-\ell}\oplus(\U_{1,2})^\ell \oplus \U_{k-\ell,k-\ell}$ for $0\leq \ell \leq \min\{k,n-k\}$. The Tutte polynomial of such a matroid is:
        \begin{align*}
            T(\U_{0,n-k-\ell}\oplus(\U_{1,2})^\ell \oplus \U_{k-\ell,k-\ell}) &= y^{n-k-\ell}\, (x+y)^\ell\, x^{k-\ell}.
        \end{align*}

    Without loss of generality, let us assume $k\leq n-k$. We want to prove that the $k+1$ polynomials corresponding to the disconnected matroids, together with $T(\U_{k,n})$ and $T(\LL_{1,k,n-k,n})$ form a family of $k+3$ linearly independent polynomials. To simplify the task, it suffices to show that the $k+3$ univariate polynomials obtained by specializing $x=y$ are themselves independent.
    
    The first $k+1$ polynomials read
    $x^{n-k-\ell}(x+x)^{\ell}x^{k-\ell} = 2^{\ell} x^{n-\ell}$ for $0\leq \ell\leq k$. 
    In particular, these $k+1$ polynomials have different degrees, all of which are at least $n-k$. 
    
    On the other hand, the specialization $x=y$ of $T(\U_{k,n})$ and $T(\LL_{1,k,n-k,n})$ has certainly a non vanishing linear term as both of the matroids are connected and hence their $\beta$-invariant is non-zero. Moreover, since $\LL_{1,k,n-h,n}$ differs from the uniform matroid $\U_{k,n}$ by a circuit-hyperplane relaxation, one has
    $T(\LL_{1,k,n-k,n})-T(\U_{k,n})=xy-x-y$ (see for example \cite[equation~(4)]{merino-tutte}). This results in $x^2-2x$ for the specialization $x=y$.
    By considering the linear and quadratic coefficients for those two polynomials, we conclude that all $k+3$ polynomials are linearly independent whenever $n>4$.
    Furthermore, Example~\ref{example:tuttes42} shows that this also holds for $n=4$.
    
    To see that these $k+3$ polynomials are indeed a basis, it remains to show that they span $T(\mathcal{N}_{n,k})$. Consider any matroid $\M$ in $\mathcal{N}_{n,k}$. We shall assume that it is connected, as otherwise its Tutte polynomial is already part of our linearly independent set. By Lemma~\ref{lem:class_N2} such matroid $\M$ is connected and sparse paving. 
    But a well-known property of the Tutte polynomial of a sparse paving matroid is that it is a linear combination of $T(\U_{k,n})$ and $xy-x-y = T(\LL_{1,k,n-k,n})-T(\U_{k,n})$; see for instance \cite[equation~(5)]{merino-tutte}.
     We deduce that the $T$-rank is $k+3$ whenever $2\leq k\leq n-k$, and by duality the full statement follows.
\end{proof}

\begin{proposition}\label{prop:class_N_universal}
 The Tutte polynomial is valuatively universal within the class $\mathcal{N}$.
\end{proposition}

\begin{proof}
    We show that the $T$-rank and $\mathcal{G}$-rank of the family $\mathcal{N}_{n,k}$ agree.
    The family $\mathcal{N}_{n,k}$ includes all matroids of rank $k$ on $n$ elements whenever $k\in\{0,1,n-1,n\}$. In particular, the $T$-ranks of 
    $\mathcal{N}_{0,n}=\mathcal{M}_{0,n}$,
    $\mathcal{N}_{1,n}=\mathcal{M}_{1,n}$,
    $\mathcal{N}_{n-1,n}=\mathcal{M}_{n-1,n}$, and
    $\mathcal{N}_{n,n}=\mathcal{M}_{n,n}$ agree with the corresponding $\mathcal{G}$-ranks (this is straightforward to see, but as an alternative we refer to Corollary~\ref{coro:T-rank-of-matroids} and Remark~\ref{remark:G-rank-of-matroids}).
    
    It remains to show that the $\mathcal{G}$-rank of $\mathcal{N}_{n,k}$ for $2\leq k\leq n-2$ is bounded from above by $\min\{k+3,n-k+3\}$, as the  $\mathcal{G}$-rank of $\mathcal{N}_{n,k}$ is bound from below by the $T$-rank of $\mathcal{N}_{n,k}$ which is $\min\{k+3,n-k+3\}$ by Proposition~\ref{prop:class_N_tutte}.
    
    By \cite[Corollary~6.7 \& Example~6.5]{ferroni-schroter} we know that for every sparse paving matroid~$\M$ of rank $k$ on $n$ elements
    there is an integer $\lambda$ such that
    \[
        \mathcal{G}(\M) = (\lambda+1)\,\mathcal{G}(\U_{k,n})-\lambda\,\mathcal{G}(\LL_{1,k,n-k,n}) .
    \]
    This implies that the $\mathcal{G}$-span of the connected matroids in $\mathcal{N}_{n,k}$ has rank at most two. Since there are in total $\min\{n-k+1,k+1\}$ disconnected matroids in $\mathcal{N}_{n,k}$, the span of the $\mathcal{G}$-invariants of all matroids $\mathcal{N}_{n,k}$ has rank at most $2+\min\{n-k+1,k+1\}$. As explained above, this yields the desired equality.
\end{proof}

\begin{remark}
    We proved that for $2\leq k \leq n-2$ both the $T$-rank and $\mathcal{G}$-rank of the family $\mathcal{N}$ is $\min\{k+3,n-k+3\}$.
    Clearly $k+3 < k(n-k)+1$ whenever $n>4$, thus the free abelian group spanned by $T(\mathcal{N}_{n,k})$ has strictly smaller rank than the one spanned by $T(\mathcal{M}_{n,k})$. Similarly, the rank span of $\mathcal{G}(\mathcal{N}_{n,k})$ is strictly smaller than the span of $\mathcal{G}(\mathcal{M}_{n,k})$. In particular, it is not possible to formulate a counterpart of Theorem~\ref{thm:cuspidals-basis-all} using this family of matroids.
\end{remark}

\subsection{Two more classes of matroids}

In this subsection we consider two more families of matroids, denoted by $\mathcal{U}$ and $\mathcal{T}$. Both are defined in a similar way, $\mathcal{U}$ is the class consisting of all uniform matroids with additional loops and coloops, whereas $\mathcal{T}$ is the class of minimal matroids with additional loops and coloops. Both of these families appear within the statement of Theorem~\ref{thm:main-characterization} as well as Theorem \ref{thm:uniform-loops-coloops-basis-all} and Theorem \ref{thm:graphicschubert-basis-all}.

\subsubsection{Uniform matroids with additional loops and coloops}

A prototypical element in $\mathcal{U}_{n,k}$ is of the form $\U_{0,m}\oplus\U_{k-\ell,n-\ell-m}\oplus\U_{\ell,\ell}$. The following provides a handy characterization of the family $\mathcal{U}$.

\begin{proposition}
    A matroid $\M$ belongs to $\mathcal{U}$ if and only if $\M$ does not have minors isomorphic to $\mathsf{T}_{2,4}$ or $\U_{1,2}\oplus \U_{1,2}$.
\end{proposition}

\begin{proof}
    It is well known that $\mathsf{T}_{2,4}$ is the single excluded minor for the class of partition matroids as this is the smallest connected matroid with three cyclic flats. On the other hand, if one further forbids the minor $\U_{1,2}\oplus\U_{1,2}$ within the class of partition matroids, it results in the class $\mathcal{U}$ that we described.
\end{proof}

\begin{proposition}\label{prop:U-tuttes-independent}
    There are exactly $k(n-k)+1$ non-isomorphic matroids in $\mathcal{U}_{n,k}$, all of whose Tutte polynomials are independent.
\end{proposition}

\begin{proof}
    We begin by counting the matroids in $\mathcal{U}$ that have rank $k$ and size $n$.
    Clearly, if $\ell = k$ or $m=n-k$ we have $\U_{0,m}\oplus\U_{k-\ell,n-\ell-m}\oplus\U_{\ell,\ell} = \U_{0,n-k}\oplus\U_{k,k}$. Furthermore, for all $0\leq \ell < k$ and $0\leq m < n-k$ are the uniform matroids $\U_{k-\ell,n-\ell-m}$ distinct and of rank and corank at least one.
    Therefore we deal with $k(n-k)+1$ non-isomorphic matroids.

    The Tutte polynomial of the elements of $\mathcal{U}_{n,k}$ are given by:
    \begin{align}\label{eq:tutte-U}
        T(\U_{0,m}\oplus\U_{k-\ell,n-\ell-m}\oplus\U_{\ell,\ell}) &= 
        \sum_{i=\ell+1}^{k} \binom{n-m-i-1}{n-m-k-1} \, x^{i}y^{m}\\ 
        &\qquad\qquad + \sum_{i=m+1}^{n-k} \binom{n-\ell-i-1}{k-\ell-1}\, x^{\ell}y^{i}\nonumber
    \end{align}
    if $0\leq \ell < k$ and $0\leq m < n-k$ and $T(\U_{0,n-k}\oplus\U_{k,k}) = x^{k} y^{n-k}$.

    Notice that the Tutte polynomial of $\U_{0,n-k}\oplus \U_{k,k}$ consists of a single monomial, which in turn does not appear in any of the other Tutte polynomials. Hence, to prove our statement it suffices to check that the first $k(n-k)$ Tutte polynomials described above are independent. 
    
    Notice that the monomials of smallest total degree in the Tutte polynomial of equation~\eqref{eq:tutte-U} are $x^{\ell+1}y^m$ and $x^{\ell}y^{m+1}$. 
    For fixed values $\ell$ and $m$, among all matroids in $\mathcal{U}_{n,k}$ whose total number of loops and coloops is exactly $s=\ell+m$. Moreover there is only one matroid having a non-zero coefficient for the monomial $x^{\ell+1}y^m$.
    Thus, there are no dependencies among these polynomials.

    Now, consider any non-trivial dependency relation of the $k(n-k)$ Tutte polynomials. Let $s$ be the smallest number such that there is a Tutte polynomial of total degree $s+1$ having a non-zero coefficient in the relation. The above paragraph implies that the coefficient accompanying this Tutte polynomial has to be zero, which is a contradiction. Hence, we conclude that the $k(n-k)$ Tutte polynomials are indeed independent.
\end{proof}

For an alternative proof of the above proposition, we refer to \cite[Theorem~5.1]{kung-syzygies}.

\subsubsection{Minimal matroids with additional loops and coloops}

Now we are ready to consider the last family of matroids, $\mathcal{T}$,  consisting on all minimal matroids with additional loops and coloops. As before, a prototypical element in $\mathcal{T}_{n,k}$ is of the form $\U_{0,m}\oplus\T_{k-\ell,n-\ell-m}\oplus\U_{\ell,\ell}$. The following characterizes $\mathcal{T}$ in several equivalent ways.

\begin{proposition}\label{prop:equivalences-T}
    The following conditions are equivalent:
    \begin{enumerate}[\normalfont (i)]
        \item $\M$ is isomorphic to a matroid of the form $\U_{0,m}\oplus\T_{k-\ell,n-\ell-m}\oplus\U_{\ell,\ell}$.
        \item $\M$ has no minor isomorphic to $\U_{2,4}$ nor $\U_{1,2}\oplus\U_{1,2}$.
        \item $\M$ is binary and Schubert.
        \item $\M$ is graphic and Schubert.
    \end{enumerate}
\end{proposition}

\begin{proof}
    (i) $\Rightarrow$ (ii) It is straightforward to check that these matroids have neither $\U_{2,4}$ nor $\U_{1,2}\oplus\U_{1,2}$ as minor.
    
    (ii) $\iff$  (iii). The only excluded minor for binary matroids is $\U_{2,4}$ while the class of Schubert matroids has infinitely many excluded minors. They are the matroids $\N_k = \trunc^{k-2}(\U_{k-1,k}\oplus\U_{k-1,k})$ for all values $k\geq 2$; see \cite[Theorem~13]{OxleyPrendergastRow:1982}.

    Let us show that $\N_k$ is not binary when $n\geq 3$. To this end, by the scum theorem it will suffice to show that $\N_k$ possesses a flat of corank $2$ contained in four or more flats of corank $1$. Since $\N_k$ is defined as the $(k-2)$-fold truncation of $\U_{k-1,k}\oplus\U_{k-1,k}$, the problem reduces to checking that the matroid $\U_{k-1,k}\oplus \U_{k-1,k}$ possesses a flat of rank $k-2$ contained in four or more flats of rank $k-1$. This is evident, because if we take any set of size $k-2$ in the first direct summand it is automatically closed (because $k\geq 3$), and adding any of the missing $2$ elements of the first summand or any of the $k$ elements of the ground set of the second direct summand we produce always a flat of rank $k-1$. In other words, we have a flat of rank $k-2$ contained in at least $2+k>4$ flats of rank $k-1$. Hence, $\N_k$ is not binary when $k\geq 3$, and therefore it contains a minor isomorphic to $\U_{2,4}$. We conclude that the excluded minors for binary Schubert matroids are precisely the uniform matroid $\U_{2,4}$ and $\N_2 = \U_{1,2}\oplus\U_{1,2}$.

    (iii) $\iff$ (iv) It is well known that Schubert matroids are transversal. Within the class of transversal matroids being graphic and being binary are equivalent conditions, see \cite[Theorem~10.4.7]{oxley}.


    (iv) $\Rightarrow$ (i). The matroid of a simple graph with two cycles has a $\U_{1,2}\oplus\U_{1,2}$ as a minor and thus is not a Schubert matroid.
    Moreover, a matroid with two distinct parallel classes has the same minor.
    In other words every connected component of a graphic Schubert matroid is a loop, coloop or minimal matroid. The direct sum of two minimal matroids is not Schubert whenever neither of them is a loop nor coloop. We conclude that all graphic Schubert matroids are of the form $\U_{0,m}\oplus\T_{k-\ell,n-\ell-m}\oplus\U_{\ell,\ell}$.
\end{proof}

\begin{proposition}\label{prop:T-tuttes-independent}
    There are exactly $k(n-k)+1$ non-isomorphic matroids in $\mathcal{T}_{n,k}$, all of whose Tutte polynomials are independent.
\end{proposition}

\begin{proof}
    The proof of the fact that there are exactly $k(n-k)+1$ non-isomorphic matroids in $\mathcal{T}_{n,k}$ is completely analogous to the one given for $\mathcal{U}_{n,k}$ in Proposition~\ref{prop:U-tuttes-independent}.
    Recall that the minimal matroid $\mathsf{T}_{k,n}$ is the cuspidal matroid $\LL_{k-1,k,k,n}$. In particular $\mathsf{T}_{k,n}$ is the circuit-hyperplane relaxation of the matroid $\U_{k-1,k}\oplus \U_{1,n-k}$. We have that the Tutte polynomial of the minimal matroid $\mathsf{T}_{k,n}$ admits the following compact expression:
    \begin{align*} 
        T(\mathsf{T}_{k,n})  &= T(\U_{k-1,k})\cdot T(\U_{1,n-k}) + x + y - xy\\
        &= \left(y + \sum_{i=1}^{k-1} x^i\right)\left(x + \sum_{i=1}^{n-k-1} y^i\right) + x + y - xy\\
        &=  \sum_{i=1}^k x^i + \sum_{i=1}^{n-k} y^i + \left(\sum_{i=1}^{k-1} x^i\right) \left(\sum_{i=1}^{n-k-1} y^i\right).
    \end{align*}

    From here, we see that the Tutte polynomial of the matroid $\U_{0,m}\oplus \mathsf{T}_{k-\ell,n-\ell-m}\oplus \U_{\ell,\ell}$ can be written as:
    \[ \sum_{i=\ell+1}^{\ell+k} x^iy^m + \sum_{i=m+1}^{m+n-k} x^{\ell}y^i + \left(\sum_{i=\ell+1}^{\ell+k-1} x^i\right) \left(\sum_{i=m+1}^{m+n-k-1} y^i\right).\]

    The proof follows by using the exact same strategy as in the proof of Proposition~\ref{prop:U-tuttes-independent}. The only matroid in our collection for which the Tutte polynomial contains a non-zero coefficient for the monomial $x^{k}y^{n-k}$ is $\U_{0,n-k}\oplus \U_{k,k}$. So, we may focus on proving that the remaining $k(n-k)$ polynomials as above are independent. Identical reasoning to the one used before yields the desired independence.
\end{proof}

\subsubsection{Universality for the two classes $\mathcal{U}$ and $\mathcal{T}$} 

The following result is a  consequence of our counting arguments in the previous subsections.

\begin{proposition}\label{prop:U_and_T_are_univ}
    The Tutte polynomial is valuatively universal within both classes $\mathcal{U}$ and $\mathcal{T}$.
\end{proposition}

\begin{proof}
    For every $0\leq k \leq n$ the number of non-isomorphic matroids in $\mathcal{U}_{n,k}$ and $\mathcal{T}_{n,k}$ is the same, and this quantity equals $k(n-k)+1$. In particular, the $\mathcal{G}$-rank of $\mathcal{U}_{n,k}$ and $\mathcal{T}_{n,k}$ is bounded from above by $k(n-k)+1$.
    
    In general, the $T$-rank is at most the $\mathcal{G}$-rank, and given that the $T$-rank of both $\mathcal{U}_{n,k}$ and $\mathcal{T}_{n,k}$ is precisely $k(n-k)+1$,
    also the $\mathcal{G}$-rank of these classes must be $k(n-k)+1$. Now, by Proposition~\ref{prop:equivalences-universality} we obtain that Tutte polynomial is valuatively universal within both of these classes.
\end{proof}

The above result will be key for the proof of Theorem~\ref{thm:main-characterization}. Moreover, in the same way that Theorem~\ref{thm:split-universal} can be used to prove Theorem~\ref{thm:cuspidals-basis-all}, the preceding proposition yields the following appealing result which is a compressed reformulation of Theorem~\ref{thm:uniform-loops-coloops-basis-all} and Theorem~\ref{thm:graphicschubert-basis-all}.

\begin{mythm}{D \& E}
    The Tutte polynomial of every matroid can be written uniquely as an integer combination of Tutte polynomials of graphic Schubert matroids and also Tutte polynomials of uniform matroids with additional loops and coloops.
\end{mythm}    



\begin{proof}[Proof of Theorems~\ref{thm:uniform-loops-coloops-basis-all} \& \ref{thm:graphicschubert-basis-all}]
    We follow a similar strategy as in the proof of Theorem~\ref{thm:cuspidals-basis-all}, but with one substantial difference. Unlike the case of elementary split matroids, within the classes $\mathcal{T}$ and $\mathcal{U}$ all the matroids are Schubert. We will indicate the proof for the family $\mathcal{T}$, because the one for $\mathcal{U}$ is almost identical.
    
    First of all, since the $T$-rank of $\mathcal{T}_{n,k}$ is $k(n-k)+1$, which by Corollary~\ref{coro:T-rank-of-matroids} coincides with the $T$-rank of $\mathcal{M}_{n,k}$, we certainly obtain that the Tutte polynomial of any matroid can be obtained as a rational combination of the Tutte polynomials of matroids in $\mathcal{T}_{n,k}$. Again, we can rely on Lemma~\ref{lemma:technical-lemma} because all the matroids in $\mathcal{T}_{n,k}$ are themselves Schubert, whereas the Tutte polynomial is valuatively universal within $\mathcal{T}_{n,k}$ by Proposition~\ref{prop:U_and_T_are_univ}. We conclude the desired integrality.
\end{proof}

It is possible to prove the above theorem, as well as Theorem~\ref{thm:cuspidals-basis-all}, using brute-force, i.e., avoiding the use of the technical Lemma~\ref{lemma:technical-lemma}.

As an immediate consequence of Theorem~\ref{thm:graphicschubert-basis-all}, we are able to solve two problems proposed in \cite[Section~9]{kung-syzygies}.

\begin{corollary}\label{coro:T-rank-binary}
    Let $\mathcal{B}$ be the class of binary matroids, and $\mathscr{Gr}$ the class of graphic matroids. The $T$-rank of $\mathcal{B}_{n,k}$ and $\mathscr{Gr}_{n,k}$ is $k(n-k)+1$.
\end{corollary}

\section{Proof of our characterization}\label{sec:main-proof}

Now we are prepared to prove our main result, Theorem~\ref{thm:main-characterization}.

\mainresult*

\begin{proof}
    Let us assume that $T$ is valuatively universal within the class $\mathcal{C}$. By the calculations in Example~\ref{example:tuttes42} and Example~\ref{example:gs42} it follows that our class $\mathcal{C}$ cannot contain all isomorphism classes of matroids on $4$ elements and rank $2$. 
    More precisely the two relations \eqref{eq:relation1}
 and \eqref{eq:relation2}, that do not have a counterpart for the universal invariant $\mathcal{G}$,
    imply that at least $\U_{0,1}\oplus \U_{1,2}\oplus \U_{1,1}$ or one of the three pairs $\{\T_{2,4},\U_{1,2}\oplus\U_{1,2}\}$, $\{\U_{2,4},\U_{1,2}\oplus\U_{1,2}\}$ or $\{\U_{0,1}\oplus\U_{2,3}, \U_{1,3}\oplus\U_{1,1}\}$ is not contained in $\mathcal{C}$.
    A~fortiori it may happen that $\mathcal{C}$ has multiple of the above or other additional excluded minors. Moreover, we want to remind the reader that the class $\mathcal{C}$ should be closed under duality, thus if $\U_{0,1}\oplus\U_{2,3}$ is excluded so must be $\U_{1,3}\oplus\U_{1,1}$ and vice versa.
    
    Let us discuss which are the four classes that result by excluding the single minor $\U_{0,1}\oplus \U_{1,2}\oplus \U_{1,1}$, or precisely one of the three listed pairs.
    
    If we exclude only $\U_{0,1}\oplus \U_{1,2}\oplus \U_{1,1}$ then we arrive at the class of all elementary split matroids by Theorem~\ref{thm:elem_split_matroids}.
    We have seen in Theorem~\ref{thm:split-universal} that this class is a family for which the Tutte polynomial is valuatively universal. We conclude that this class is inclusion-wise maximal with that property, and further that the matroid $\U_{0,1}\oplus \U_{1,2}\oplus \U_{1,1}$ is contained in every other maximal class.

    If we exclude the minimal matroid $\T_{2,4}$ and the direct sum $\U_{1,2}\oplus\U_{1,2}$ we are in the situation of Proposition~\ref{prop:U-tuttes-independent} and consider the class $\mathcal{U}$ of uniform matroids with additional loops and coloops.
    This class again is a family for which the Tutte polynomial is valuatively universal by Proposition~\ref{prop:U_and_T_are_univ}.

    Similarly, if we exclude the uniform matroid $\U_{2,4}$ and the direct sum $\U_{1,2}\oplus\U_{1,2}$ we arrive at the class $ \mathcal{T}$ of minimal matroids with additional loops and coloops. Proposition~\ref{prop:U_and_T_are_univ} again reveals that this is a class for which the Tutte polynomial is valuatively universal.

    The last class we have to consider is the class $\mathcal{N}$ we studied in Section~\ref{sec:the_class_N}.
    Theorem~\ref{prop:class_N_universal} tells us that within this class the Tutte polynomial is valuatively universal, too. Among the duality closed classes, the class $\mathcal{N}$ is clearly maximal with that property.

    Notice that every class $\mathcal{C}$ closed under minors and duals and for which the Tutte polynomial is valuatively universal has to be a subclass of one of the four listed. However, the four listed class are indeed valid. Hence, we conclude that they are maximal and moreover there are no other maximal classes.
\end{proof}

\section{Final remarks and some open problems}\label{sec:open_problems}

\subsection{Analogs for other valuative invariants}

Throughout this article we have focused on the Tutte polynomial, motivated by its ubiquity and remarkable features. However, there are many other valuative invariants for which one could ask for a counterpart of our Theorem~\ref{thm:main-characterization} or find an invariant for a given class of matroids whose $\mathcal{G}$-invariants span a subgroup. 

\begin{problem}
    Prove counterparts of Theorem~\ref{thm:main-characterization} for other valuative invariants.
\end{problem}

As we have seen, the most interesting class for which the map $T$ is a universal valuative invariant is precisely that of elementary split matroids. It is tempting to ask whether other valuative invariants, for example the spectrum polynomial \cite{kook-reiner-stanton}, may lead to further interesting classes of matroids. The valuativity of the spectrum polynomial was established in \cite[Corollary~9.26]{ferroni-schroter}. We mention this particular invariant for two reasons: i) it is known that it is neither a specialization of the Tutte polynomial nor does it specialize to the Tutte polynomial, and ii) the spectrum polynomial of a matroid $\M\in\mathcal{M}_{n,k}$ contains at most $nk$ non-zero coefficients (in particular, it contains much less data than the $\mathcal{G}$-invariant, which may contain up to $\binom{n}{k}$ non-zero coefficients).

\subsection{Dropping the assumption on dual-closedness}

On a separate note, a shrewd reader may ask if it is possible to drop the condition on the class $\mathcal{C}$ being closed under duality in the statement of Theorem~\ref{thm:main-characterization}. This is a question that deserves to be addressed.

\begin{problem}
    Find a counterpart of Theorem~\ref{thm:main-characterization} dropping the assumption on $\mathcal{C}$ being closed under duality.
\end{problem}

Although we believe the structure of a possible proof could be more or less similar to what we have done in this paper, we expect that it will be necessary to deal with a huge number of cases. For instance, if one takes $\mathcal{C}_{4,2}=\mathcal{M}_{4,2}\smallsetminus\{\U_{0,1}\oplus \U_{2,3}\}$, and similarly if one only removes the dual of that matroid, the $\mathcal{G}$-rank and the $T$-rank of $\mathcal{C}_{4,2}$ coincide. However, if one considers the $\mathcal{G}$-rank and $T$-rank of the matroids in $\mathcal{M}_{5,2}$ that avoid the minor $\U_{0,1}\oplus \U_{2,3}$, there are $10$ maximal subsets of $\mathcal{M}_{5,2}$ for which the $\mathcal{G}$-rank equals the $T$-rank. From there one can only hope that the subsequent ramification into further cases eventually collapses, as it indeed happened in the main proof of the present article. Moreover, it is unclear if the maximal classes would only have a \emph{finite} number of excluded minors.

\begin{question}
    Let $\mathcal{C}$ be an inclusion-wise maximal minor-closed family of matroids for which the Tutte polynomial is valuatively universal. Does $\mathcal{C}$ have a finite list of forbidden minors?
\end{question}

\subsection{The graphic Schubert coefficients}

By virtue of Theorem~\ref{thm:graphicschubert-basis-all}, we know that the Tutte polynomial of every matroid can be calculated uniquely as a signed integer sum of Tutte polynomials of very special graphs. Is it possible to interpret these coefficients in some meaningful or combinatorial way?

\begin{problem}
    Understand the coefficients of the decomposition of Theorem~\ref{thm:graphicschubert-basis-all}.
\end{problem}

It is also reasonable to formulate analogous questions for the basis of cuspidal matroids, or the basis of uniform matroids with additional loops and coloops. 

\subsection{Tutte ranks for combinatorial geometries}

As we have seen in Corollary~\ref{coro:T-rank-of-matroids}, the $T$-rank of $\mathcal{M}_{n,k}$ is $k(n-k)+1$. If we consider the class of \emph{simple} matroids (also known as combinatorial geometries) on $n$ elements and rank $k$, it is still possible to calculate the $T$-rank using a similar proof. Precisely, one obtains that this equals $(k-2)(n-k)+1$ for $2\leq k\leq n$. An intriguing open question suggested to us by Kung in a private communication is to compute the $T$-rank of the class of simple graphic (or binary) matroids on $n$ elements and rank $k$.

\begin{problem}\label{problem:simple-binary-T-rank}
    Determine the $T$-rank for the class of graphic and binary simple matroids on $n$ elements and rank $k$.
\end{problem}

Unfortunately, our proof for Corollary~\ref{coro:T-rank-binary} does not extend, and it does not seem to be straightforward to adapt. Moreover, our computer experiments of calculating these $T$-ranks do not reveal useful patterns to us for small values of $k$ and $n$.  

\subsection{Some remarks on Tutte polynomials of elementary split matroids}

Before ending, let us digress about the behavior of Tutte polynomials for elementary split matroids, in contrast to their behavior for arbitrary matroids. Although Theorem~\ref{thm:cuspidals-basis-all} says that the Tutte polynomial of any matroid $\M$ can be written uniquely as a combination of Tutte polynomials of cuspidal matroids, it remains broadly open to understand better the coefficients of this linear combination. On the other hand, in our previous work \cite[Example~7.21]{ferroni-schroter} we explained that Tutte polynomials cannot identify whether a given matroid is or is not elementary split. Despite these observations, the reader should not be misled to think that Tutte polynomials of arbitrary matroids and elementary split matroids are ``equally complicated''. A remarkable instance of the difference in complexity is the fact that for elementary split matroids we have been able to establish the Merino--Welsh conjecture \cite{ferroni-schroter-mw}, i.e., we proved that for $\M$ elementary split, loopless and coloopless, the inequality $T_{\M}(2,0)\,T_{\M}(0,2)\geq T_{\M}(1,1)^2$ holds. However, a recent breakthrough by Beke, Cs\'aji, Csikv\'ari and Pituk \cite{beke-csaji-csikvari-pituk} establishes the existence of a loopless and coloopless matroid (which is not elementary split) not satisfying the inequality predicted by Merino and Welsh. 

\section*{Acknowledgements}
\noindent We want to acknowledge insightful exchanges with Joseph Kung, to whom we thank for his comments on an earlier version of this manuscript. In particular, Problem~\ref{problem:simple-binary-T-rank} was incorporated upon his suggestion. We are also thankful to Chris Eppolito for a careful reading and useful comments that improved the exposition of this article.

\bibliographystyle{amsalpha}
\bibliography{bibliography}

\end{document}